\documentclass[english, 11pt]{article}
\usepackage[latin9]{inputenc}
\usepackage{geometry}
\usepackage{color}
\usepackage{verbatim}
\usepackage{amsmath}
\usepackage{amssymb}
\usepackage{esint}
\usepackage{bm}
\usepackage[round, colon, authoryear]{natbib}

\makeatletter
\@ifundefined{definecolor}{\usepackage{color}}{}
\@ifundefined{definecolor}{\usepackage{color}}{}
\@ifundefined{definecolor}{\usepackage{color}}{}
\@ifundefined{definecolor}{\usepackage{color}}{}
\@ifundefined{definecolor}{\usepackage{color}}{}
\@ifundefined{definecolor}{\usepackage{color}}{}
\@ifundefined{definecolor}{\usepackage{color}}{}
\@ifundefined{definecolor}{\usepackage{color}}{}
\@ifundefined{definecolor}{\usepackage{color}}{}
\@ifundefined{definecolor}{\usepackage{color}}{}
\@ifundefined{definecolor}{\usepackage{color}}{}
\@ifundefined{definecolor}{\usepackage{color}}{}
  
\usepackage{amsthm}

\@ifundefined{definecolor}{\usepackage{color}}{}
\usepackage{babel}

\usepackage{xr}
\newtheorem{theorem}{Theorem}[section]

\newtheorem{corollary}{Corollary}[section]

\newtheorem{definition}{Definition}[section]

\newtheorem{lemma}{Lemma}[section]
\newtheorem{remark}{Remark}[section]

\numberwithin{equation}{section}

\def\keywords{\vspace{.5em}
{\textit{Keywords}:\,\relax%
}}

\newcommand{\R}{\mathbb{R}}
\newcommand{\E}{\mathbb{E}}
\newcommand{\N}{\mathbb{N}}
\newcommand{\Prob}{\mathbb{P}}

\newcommand{\1}{\mathbf{1}}
\newcommand{\CF}{\mathcal{F}}

\allowdisplaybreaks

\makeatother

\usepackage{babel}
\begin{document}

\title{Convergence in Models with Bounded Expected Relative Hazard Rates%
\thanks{We thank J.~\foreignlanguage{esperanto}{Hedlund}, S.~\foreignlanguage{british}{Lakshmivarahan}
and M.A.L.~Thathachar for helpful correspondence on the subject matter
of this paper, and two referees and the Associate Editor for insightful and constructive comments. Oyarzun acknowledges financial support of the Ministerio
de Ciencia y Tecnologia, FEDER funds under project SEJ2007-62656,
and of the Instituto Valenciano de Investigaciones. 
Ruf acknowledges financial support of the  Visitors Program of the School of Economics at the University of Queensland and  of the Oxford-Man Institute of Quantitative Finance at the University of Oxford, where a major part of this work was completed.
}}

\author{Carlos Oyarzun%
\thanks{School of Economics, University of Queensland. E-Mail: c.oyarzun@uq.edu.au %
}
  \and Johannes Ruf%
\thanks{Department of Mathematics, University College London. E-Mail: j.ruf@ucl.ac.uk%
} 
}
\maketitle
\begin{abstract}
We provide a general framework to study stochastic sequences related
to  individual learning in economics, learning automata in computer sciences,
social learning in marketing, and other applications. More precisely, we study
the asymptotic properties of a class of stochastic sequences that
take values in $[0,1]$ and satisfy a property called ``bounded
expected relative hazard rates.'' Sequences that satisfy this property and feature ``small step-size" or ``shrinking step-size"   converge to $1$ with high probability or almost surely, respectively. These convergence results yield conditions for the learning models in \citet{ErevRoth1998}, \citet{Schlag1998}, and \citet{Borgersetal2004}
 to choose  expected payoff maximizing actions with probability one in the long run.
\end{abstract}


\section{Introduction}

Stochastic sequences arising in the analysis of several models in economics often exhibit expected hazard rates that are proportional to
the sequence's current value.  For instance, models of technology adoption often
satisfy that the change in the fraction of a population that adopts
a new technology is proportional to the product of the current fraction
of adopters and the current fraction of non-adopters (see, e.g., \Citet{Young2009}).
This follows from the assumption that diffusion of technology requires non-adopters to observe
adopters in order to learn about the new technology.  
A similar reasoning applies to models in other disciplines, such as Bass' celebrated model  of new product growth (see, e.g., \Citet{Bass1969}, \citet{JacksonYariv2011}) and selection models in biological evolution (see, e.g., \Citet{Nowak2006}).   As we discuss below,  models of individual and social learning provide another class of examples for stochastic sequences with expected hazard rates that are proportional to the sequences' current value.  In these models, the sequences represent the probability of choosing optimal actions.

The analysis of such models   usually concerns the question whether a new technology or a product gets fully adopted, a certain type takes over in a biological selection process, or an optimal action is played almost surely in the long run.
Towards this end, this paper provides general conditions on expected hazard rates of a bounded stochastic sequence that guarantee the convergence to the upper bound.  Here,  the sequence is interpreted as a fraction of a certain type or the probability of playing an optimal action at any point in time. This paper thus provides conditions that guarantee that, in the long run, a certain type takes over the whole group of types or  only optimal actions are chosen, as illustrated in the applications discussed below.

It turns out that constraints on the \emph{relative hazard rates} of a stochastic sequence, i.e., the proportions of the hazard rates to the values of the sequence,\footnote{Formally, if the values of the sequence are denoted
by $\{P_{t}\}_{t \in \N_0}$, then the corresponding relative hazard rates are defined as $(P_{t+1}-P_{t})/((1-P_{t})P_{t})$ for all $t \in \N_0$.}  provide helpful conditions for the convergence to the upper bound.  In contrast to the deterministic case, in a stochastic framework, lower bounds for the relative hazard rates are not sufficient for almost sure convergence. For example, in the case of technology adoption, full adoption might fail as the new technology may be completely abandoned at some point in time by chance, or adoption rates may drop too fast.  
The analysis below reveals that if the underlying submartingale moves in small or shrinking steps, convergence to the upper bound holds, nevertheless. Thus, in the long run, new technologies are used or optimal actions chosen if adoption or learning occurs in small or shrinking steps.

The first main result of this paper, Theorem~\ref{lowerbound}, analyzes the asymptotical
properties of a sequence that changes with \emph{small step-size} and  satisfies weak bounds on its relative hazard rates.
Theorem~\ref{lowerbound} asserts that the probability of  \emph{convergence to optimality},
i.e., the event that the stochastic sequence converges to the upper bound, is arbitrarily
high for sequences with sufficiently small step-size. This result  allows us
to obtain novel convergence results in different contexts, including,
for instance, the models of individual and social learning that we discuss below. A limitation of Theorem~\ref{lowerbound}  is that the question of how small the step-size needs to be in order to achieve any given  probability of convergence to $1$   is usually directly related to the probability measure of the underlying probability space. In applications, however, this probability measure  is assumed to be unknown. This issue is addressed
by Theorem~\ref{T_RE_Theory} and Corollary~\ref{C as}, which provide
sufficient conditions for achieving convergence to optimality almost surely under 
an extra condition that may be interpreted as requiring an arbitrary
\emph{shrinking step-size} over time. 

These results can be applied to the analysis of several models of boundedly rational learning (see, e.g., \Citet{ErevRoth1998}, \Citet{Schlag1998}, \Citet{Borgersetal2004}). In models of individual learning,  in every period individuals choose one action out of a finite set and observe a payoff realization yielded by the action they choose (sometimes along with forgone payoffs). In models of social learning, individuals also observe the payoffs from the actions chosen by a  sample of other individuals. Learning is assumed to be ``adaptive,'' i.e., in every period, individuals make their choice  according to a probability distribution over actions and this distribution is revised as new payoff observations arrive. As discussed in Section~\ref{SS applications}, our results can be used to provide conditions for learning to yield convergence to choose expected-payoff maximizing actions, either with high probability or almost surely. All details are provided in the online appendix (\citet{OyarzunRuf2014a}).  



Small and shrinking  step-size appear often in applications.  Small step-size  has been used in both theoretical and experimental work in economics (see, e.g., \Citet{BorgersSarin1997} and \Citet{VanHuyck_etal2007}, respectively). Shrinking step-size appears endogenously in the Roth-Erev model (see, e.g., \citet{ErevRoth1998}).\footnote{Polya-urn-based schemes, akin to the Roth-Erev learning model (and  hence yielding decreasing step-size), are useful to study different allocation problems. For instance, \Citet{Durham_etal1998} and \Citet{LaruellePages2013} apply these models to the study of patient allocation in clinical trials.}    Researchers using  the \Citet{Cross1973} model in applications often assume shrinking step-size   (see, e.g., \Citet{SarinVahid2004}
), even though the benchmark version of this model has a fixed step-size. 
The condition of 
shrinking step-size captures the ``power law of practice''
in learning (see, e.g., \citet{ErevRoth1998} and the references therein): initial periods typically  exhibit a substantial response of behavior to experience and are followed by gradually decreasing responses, such as those implied by shrinking step-size. 

The question then arises when and why the ``power law of practice'' is relevant. Psychologists have  long studied this problem. For instance, \Citet{Bills1934} and \Citet{NewellRosenbloom1981} study the  decrease over time of motivation,  psychophysical performance, or cognitive gains from experience, as possible explanations of the ``power law of practice.'' These explanations  have appeal in the analysis of  economic applications, as well. In particular, motivation,  psychophysical performance, and cognitive gains play an important role in the analysis of data in experimental economics, where subjects tire and lose concentration. More importantly,  in real-world economic problems, the ``power law of practice'' seems to hold for similar reasons. 
For instance,  \Citet{Choi_etal2009} analyze reinforcement learning and saving behavior, and provide evidence supporting the ``power law of practice'' hypothesis: younger investors are more responsive to their personal return realizations than older investors in terms of their 401(k) saving rates. We believe the ``power law of practice'' plays a role in  individual and social learning in economics.

\bigskip{}

\textbf{Related literature.} \Citet{Norman1968b} formally analyzes
a two-armed bandit algorithm to study the asymptotic properties of reinforcement learning models considered by experimental psychologists (see, e.g., \Citet{weinstock1965probability}) who study learning when success or failure are the only  possible outcomes. In this pioneering work, he shows  that certain  learning models   converge with high probability to choose the action that is more likely to yield 
success, provided that changes in the probability of choosing
each action are small. Computer scientists (see, e.g., \Citet{ShapiroNarendra1969}, \citet{NarendraThathachar1974}, \Citet{LT1976}, and \Citet{Torkestani}), provide similar results in the context of learning automata. \Citet{OyarzunSarin2006} adapt these techniques   to prove
convergence of a class of learning models to risk averse choice. The settings in these papers
are more restrictive than in this work, and their convergence results are implied by Theorem~\ref{lowerbound} below.
None of these papers has a counterpart to the almost-sure
convergence results in Theorem~\ref{T_RE_Theory} and Corollary~\ref{C as} below,
as the models they analyze fail to satisfy our conditions on shrinking step-size over time.

The paper closest to our   analysis is that of \citet{Lamberton_etal2004}, who thoroughly analyze  the asymptotical properties of the two-armed bandit algorithm. This analysis is of particular interest because the algorithm may have a positive probability of converging to a non-optimal state, i.e., a ``trap,'' despite of the probability of choosing an optimal action being a submartingale.
 \citet{Lamberton_etal2004} take an  approach similar to ours based on shrinking step-size to provide conditions that yield convergence to optimality almost surely. Their analysis is
tailored to the specific characteristics of the two-armed bandit algorithm, whereas this paper's framework allows us to apply its results in more general settings such as the  models in economics that we study in the applications.

\section{Convergence  for updating rules with small and shrinking step-size}  \label{S:conv}
\subsection{Framework}  \label{SS Framework}

In this subsection, we provide the analytical framework and
introduce the condition of bounded expected relative hazard rates.


In our applications to models of individual and social learning, the realization of the state of the world in each period determines the action chosen by each individual, the obtained and
forgone payoffs, and the information revealed to each individual. After observing this information, individuals adjust their behavior, i.e., the probability of choosing each action according to their behavioral rule. We now provide a formal model that encompass these models. First, we introduce the probability space and the states of the world. Then we introduce the \emph{updating rule}, which is a function that maps the observable part of the state of the world  to the current value of a stochastic process. This process represents  performance, i.e., the probability of choosing optimal actions.  The applications discussed  in Section~\ref{SS applications} illustrate that (a slight generalization of) this setup is  broad enough to accommodate an array of models of learning.   

\bigskip

\textbf{(a) Probability space.} The possible states of the world are
represented by the measurable product space $(\Omega,\mathcal{F})=(\prod_{t=1}^{\infty}\Omega_{t},\otimes_{t=1}^{\infty}\mathcal{F}_{t})$,
where $\Omega_{t}$ stands for the set of states that may occur at
time $t\in\mathbb{N}$ and is equipped with a sigma algebra
$\mathcal{F}{}_{t}$, describing the set of events. Furthermore, let $\Omega_{[0,t]}:=\prod_{\tau=1}^{t}\Omega_{\tau}$ denote the set of all histories up to time $t$ and let $\Omega_{[0,0]}$ be an arbitrary singleton. Analogously, let 
$\mathcal{F}_{[0,t]}:=\otimes_{\tau=1}^{t}\mathcal{F}_{\tau}$ denote
the set of all events that may occur up to time $t\in\mathbb{N}$
and set $\mathcal{F}_{[0,0]}:=\{\emptyset,\Omega\}$.  Let $\mathbb{P}$ be a probability measure on $\left(\Omega,\mathcal{F}\right)$, and $\mathbb{P}_{t}$ its conditional version given $\mathcal{F}_{[0,t]}$, along with its conditional expectation $\mathbb{E}_{t}[\cdot]$.

\bigskip{}
\textbf{(b) Updating rule.} An updating
rule is  a sequence $\Pi=\left\{ \Pi_{t}\right\} _{t\in\mathbb{N}}$
of functions $\Pi_{t}:\Omega\times [0,1]\rightarrow[0,1]$ such that $\Pi_t(\cdot,p)$ is  $\mathcal{F}_{[0,t]}$--measurable for all $t\in\mathbb{N}$ and $p\in[0,1]$.
We shall usually omit the first argument of $\Pi_{t}$, for sake of notation.  For a given $P_0 \in [0,1]$,  the sequence $P = \{P_t\}_{t \in \mathbb{N}_0}$, defined via the iteration $P_{t} = \Pi_{t}(P_{t-1})$ for all $t \in \mathbb{N}$, is called the \emph{performance measure} (corresponding to the initial value $P_0$ and the updating rule $\Pi$).
\bigskip

The performance measure changes in a probabilistic way over time,
according to the information that becomes available and to the updating rule. Since the updating rule is allowed to depend on all information revealed up to that period, the performance measure may depend on the whole sequence of realized states up to that date.

In order to get an intuitive idea of the class of sequences studied in this paper, let us recall the interpretation borrowed from the literature on technology adoption mentioned in the introduction.
Let $P_{t}\in[0,1]$ denote the fraction of a continuum population that has adopted a new
technology by time $t\in\mathbb{N}_{0}$.
The change in the fraction of the
population that adopts the new technology is then assumed to be proportional to the rate at which a non-adopter observes an adopter. More precisely, if each individual  observes only one other individual, uniformly picked from the population, then the
fraction of time $t$ non-adopters who observe time
$t$ adopters is $P_{t}\left(1-P_{t}\right)$. Only a fraction of these non-adopters
are assumed to adopt the new technology at time $t+1$. In the spirit of
survival theory,  \Citet{Young2009} calls the  ratio $\left(P_{t+1}-P_{t}\right)/\left(P_{t}\left(1-P_{t}\right)\right)$
the \emph{relative hazard rate} of $P$ at time ${t\in\N_{0}}$. 

 The relative hazard rates of the models studied below
are allowed to be stochastic. This paper provides conditions on
the expected values of the relative hazard rates that guarantee that $P$ converges
to $1$ (with a high probability or almost surely).  Towards this end, we next introduce an important property for expected relative hazard rates:

\begin{definition} \label{bhr} An updating rule $\Pi$
satisfies the \emph{weakly bounded expected relative hazard rates
property (WBERHR)} with lower bound sequence $\delta:=\{\delta_{t}\}_{t\in\mathbb{N}_{0}}$,
where $\delta_{t}\geq0$ is $\mathcal{F}_{[0,t]}$--measurable, if
\begin{align}
\mathbb{E}_{t}[\Pi_{t+1}\left(p\right)]-p\geq\delta_{t}\cdot p \left(1-p\right)\label{so1}
\end{align}
 for all $p\in[0,1]$
and $t\in\mathbb{N}_{0}$. The updating rule $\Pi$
satisfies the \emph{bounded expected relative hazard rates property
(BERHR)} if it satisfies WBERHR with lower bound sequence $\delta:=\{\delta_{t}\}_{t\in\mathbb{N}_{0}}$
such that $\inf_{t\in\N_{0}}\{\delta_{t}\}>0$.\footnote{We emphasize that the BERHR property does not require that $\inf_{t\in\N_{0}}\{\delta_{t}\}$
is uniformly (in $\omega\in\Omega$) bounded away from zero.}
 \end{definition}

 The nature of phenomena that yield sequences satisfying BERHR or WBERHR is diverse. These  conditions often hold in learning models, such as those discussed  in Section~\ref{SS applications}. They also arise in   biological evolution models (see, e.g., \Citet{Nowak2006}), where the variable of interest may be the relative size of the population of a certain type of cells with respect to the whole population.\footnote{Consider the populations of two types of   cells, whose sizes at time $t\in\mathbb{N}_0$ are $X_t$ and $Y_t$ and whose exogenously given growth rates are  $a\geq-1$ and $b\geq-1$. Thus $X_{t+1}-X_t=aX_t$ and $Y_{t+1}-Y_t=bY_t$.  If $a>b$, then the fraction of the first type of cells in the population $P_t:=X_t/(X_t+Y_t)$ satisfies BERHR: $(P_{t+1}-P_t)/(P_t(1-P_t))\geq (a-b)/\left(1+a\right)$. In this case, since $P_t$ is deterministic, standard arguments yield that $\lim_{t\uparrow\infty}P_t=1$.}


The two-armed bandit algorithm in \citet{Lamberton_etal2004} satisfies WBERHR. The analysis in this paper, however, requires a more general setup than theirs in order to accommodate the applications  below. For instance, if the two-armed bandit algorithm is modified to allow that the outcomes of both arms are observed (while perhaps some  past outcomes are forgotten), then the results in \citet{Lamberton_etal2004} do not directly apply. In contrast, in the example of learning with full information that we analyze in Section~\ref{SS:individual} of the online appendix, we illustrate how to use the results provided here in that setup.

\bigskip
\textbf{(c) Convergence to optimality.} 
If the updating rule $\Pi$
satisfies BERHR or WBERHR, then the corresponding performance measure $P$ is a bounded submartingale and hence, there exists (almost surely) the random variable $P_{\infty} = \lim_{t\uparrow\infty}P_{t}$.
We call the event $\{P_{\infty}=1\}$ convergence to  optimality. The notion of almost sure  convergence to optimality also appears as ``infallibility'' in the literature (see, e.g., \citet{Lamberton_etal2004}).

\bigskip

We conclude this subsection with a relatively standard observation (see, e.g., \citet{Norman1968b}):
\begin{lemma} \label{support}If the updating rule $\Pi$
satisfies WBERHR with lower bound sequence $\delta=\{\delta_{t}\}_{t\in\mathbb{N}_{0}}$
and if $\sum_{t=0}^{\infty}\delta_{t}=\infty$, then $P_{\infty}\in\{0,1\}$.
\end{lemma}

\begin{proof} Assume that $\{P_{t}\}_{t \in \N_0}$ does not almost surely converge
to either $0$ or $1$. Then, there exists an $\varepsilon>0$ such that
$\Prob(\lim_{t\uparrow\infty}P_{t}(1-P_{t})>2\varepsilon)>2\varepsilon$.
Thus, there exists a $t_{0}\in\N$ such that the event $B:=\{P_{t}(1-P_{t})>\varepsilon\text{ for all }t\geq t_{0}\}$
satisfies $\Prob(B)>\varepsilon$.

By the hypothesis, 
\[
\mathbb{E}_{t}[P_{t+1}]-P_{t}\geq\delta_{t}P_{t}(1-P_{t})
\]
 for all $t\in\mathbb{N}_{0}$; thus, 
\[
1\geq\mathbb{E}[P_{t}]=P_{0}+\underset{\tau=0}{\overset{t-1}{\sum}}\mathbb{E}\left[P_{\tau+1}-P_{\tau}\right]\geq\overset{t-1}{\underset{\tau=0}{\sum}}\mathbb{E}\left[\delta_{\tau}P_{\tau}(1-P_{\tau})\right]\geq\underset{\tau=t_{0}}{\overset{t-1}{\sum}}\mathbb{E}\left[\1_{B}\delta_{\tau}P_{\tau}(1-P_{\tau})\right] \uparrow \infty,
\]
 as $t$ tends to infinity, leading to a contradiction. \end{proof}

\subsection{Example for the lack of almost sure convergence to optimality}

The following example, adapted from \Citet{ViswanathanNarendra1972}, illustrates
that even updating rules that satisfy a strong version of BERHR, so that the lower bound sequence
is uniformly (in $\omega\in\Omega$) bounded away from zero, and that
have a performance sequence $P$ that never gets absorbed, i.e., $P_{t}\in(0,1)$ for all $t\in\N_{0}$, may not achieve convergence to  optimality almost surely. 
Consider a two-armed bandit algorithm that, at each time $t$,
chooses one out of two arms and observes the realization of a failure
or success. Arm~1 and arm~2 succeed with probability $\mu_{1}$ and
$\mu_{2}$, respectively, where $0<\mu_1<\mu_2<1$. The probability of choosing arm~1 at time $t$ is $1-P_t$ and
the probability of choosing arm~2 is $P_t$. If at time $t$ the
observed realization is a failure, then  $P_{t+1}=P_t$. Observed successes, however, increase the probability of choosing the
same arm in the next period: if  arm~2 is chosen at time $t$ and a success is observed,
then  $P_{t+1}-P_t=(1-P_t)(1-\beta)$; and, if arm~1 is chosen and a success is observed,
then  $P_{t+1}-P_t=- P_t(1-\beta)$, for some exogenously given constant $\beta\in[0,1)$.\footnote{Equivalently, if arm~1 is chosen and a success is observed,
then  $P_{t+1}= \beta P_t$.}
Although the underlying updating rule verifies BERHR with
constant lower bound $(1-\beta)(\mu_{2}-\mu_{1})>0$, it is argued below that, with strictly positive probability, arm~1
is chosen at all times. Therefore, the event $\{P_{\infty}=0\}$
has positive probability. 

To see how convergence to choose arm~2 may fail, observe that if $\beta = 0$, then $P_1 = 0$ with probability $(1-P_0) \mu_1$. Thus, in this case, $\mathbb{P}(P_{\infty}=0) \geq (1-P_0) \mu_1$.  If $\beta \in (0,1)$, partition the set of time periods  in subsets or \emph{blocks} of consecutive time periods, with cardinalities $1,2,3,...$, i.e., $\{1\}, \{2,3\},\{4,5,6\}, ....$ Now, fix $j>1$ and consider the event in which  arm~$1$  is chosen at all
times until the last time of the  $(j-1)^{th}$ block, with at least
one success in each block.  Conditional on that event, the  probability of choosing arm~$1$ at all times in the
block of length $j$ is at least $(1-\beta^{j-1}P_0)^{j}$. Furthermore, the probability of obtaining
at least one success in the $j^{th}$ block, given that arm~$1$
is chosen  in all time periods in that block, is $1-(1-\mu_{1})^{j}$. Therefore, the probability that arm~$1$ is chosen at  all times until the last time period of the $N^{th}$ block,  and that at least one success occurs in each block is at least 
\begin{equation}
\prod_{j=1}^{N}(1-\beta^{j-1}P_{0})^{j}\cdot\prod_{j=1}^{N}(1-(1-\mu_{1})^{j}).\label{eq:vn_lowerbound}
\end{equation}

We now argue that the limit of the expression in \eqref{eq:vn_lowerbound},
as $N$ tends to $\infty$, is strictly positive, which, directly
yields $\Prob(P_{\infty}=0)>0$. Since $\sum_{j=1}^{\infty}\left(1-\mu_{1}\right)^{j}<\infty$,
we have  $\prod_{j=1}^{\infty}(1-(1-\mu_{1})^{j})>0$.%
\footnote{We use that, for $a_{j}<1$, the product $\prod_{j=1}^{\infty}(1-a_{j})$ converges
to a strictly positive number if the sum  $\sum_{j=1}^{\infty}a_{j}$ converges
absolutely, which follows from taking logarithms and the limit comparison
test.%
} Finally, the inequalities $(1-\beta^{j-1}P_0)^{j}\geq1-j\beta^{j-1}P_0>0$
for all large enough $j$ imply that the first product in \eqref{eq:vn_lowerbound}
converges to a strictly positive number since $\sum_{j=1}^{\infty}j\beta^{j-1}<\infty$
by the ratio test.

Intuitively, when an action is chosen initially and sufficient
successes are observed, there is a positive probability that this
action is always chosen.
However, 
Corollary~\ref{C as} below reveals that this event would not occur if the updating rule had decreasing step-size. In this example, such an updating rule is obtained by replacing $(1-\beta)$ with $(1-\beta)/(t+2)$  when updating $P_t$ to $P_{t+1}$ for all $t \in \N_0$.
Theorem~\ref{lowerbound} yields an analog statement.

As this  example illustrates, the BERHR property does not guarantee convergence to optimality almost surely. This paper is mainly concerned with strengthening this property to obtain convergence to  optimality  with high probability or almost surely.

\subsection{Convergence results}
\label{SS convergence}

In this subsection, we analyze the asymptotic properties of the performance
measure of updating rules that satisfy either BERHR or WBERHR. In particular,
we provide sufficient conditions for  updating rules that satisfy these properties with shrinking step-size,
in a sense that we make precise below, to yield convergence to optimality almost
surely. 


Towards this end, we fix an updating rule $\Pi$ and  a sequence $\theta:=\{\theta_{t}\}_{t\in\mathbb{N}_{0}}$
such that $\theta_{t}\in(0,1]$ is $\mathcal{F}_{[0,t]}$--measurable, which we call a \emph{compressing sequence}. 
We now consider the updating rule  $\Pi^{\theta}=\left\{ \Pi_{t}^{\theta}\right\} _{t\in\mathbb{N}}$ given by the sequence of functions $\Pi_{t}^{\theta}:\Omega\times [0,1]\rightarrow[0,1]$
 that satisfy 
\begin{equation}
\Pi_{t}^{\theta}(\cdot,p):=p+\theta_{t-1}\left(\Pi_{t}\left(\cdot,p\right)-p\right)\label{eq:slowdef}
\end{equation}
for all $p\in[0,1]$ and $t\in\mathbb{N}.$ 

We say that
the updating rule $\Pi^{\theta}$ is a \emph{small step-size version} of $\Pi$. Let $P^{\theta}=\left\{ P_{t}^{\theta}\right\} _{t\in\mathbb{N}_{0}}$ be the corresponding \emph{small step-size version} of $P$;  to wit, $P_0^{\theta}=P_0\in[0,1]$ and
$P_t^{\theta}=\Pi_{t}^{\theta}(P_{t-1}^{\theta})$
for all $t\in\mathbb{N}$.  If $\Pi$ satisfies WBERHR  with lower bound sequence $\delta$, then $\Pi^{\theta}$
satisfies WBERHR with lower bound sequence $\{\theta_{t}\delta_{t}\}_{t\in\mathbb{N}_{0}}$ since
\begin{align} \label{eq:slowberhr}
\mathbb{E}_{t}\left[\Pi_{t+1}^{\theta}\left(p\right)\right]-p=\mathbb{\theta}_{t}\mathbb{E}_{t}[\Pi_{t+1}\left(p\right)-p]\geq\theta_{t}\delta_{t}p(1-p)
\end{align}
for all  $p\in[0,1]$
and $t\in\N_{0}$. Furthermore, if $\Pi$
satisfies BERHR and $\inf_{t\in\mathbb{N}_{0}}\{\theta_{t}\}>0$,
$\Pi^{\theta}$ satisfies BERHR as well.
Hence, as before, we can define $P_{\infty}^{\theta}:=\lim_{t\uparrow\infty}P_{t}^{\theta}$
for $\Pi^{\theta}$.
Compared to the updating rule $\Pi$ with corresponding performance measure $P$, the updating rule $\Pi^\theta$ yields a performance measure $P^\theta$ that moves in the same direction as $P$, but  a smaller magnitude.

We are now ready to state the first result of this paper:
\begin{theorem} \label{lowerbound} Suppose the updating rule $\Pi$
satisfies WBERHR with lower bound sequence $\delta=\{\delta_{t}\}_{t\in\mathbb{N}_{0}}$,
$\sum_{t=0}^{\infty}\delta_{t}^{2}=\infty$, and $P_0>0$.
Then, for all $\varepsilon>0$, the sequence $\theta=\{\theta_{t}\}_{t\in\N_{0}}$
with $\theta_{t}:=(1\wedge\delta_{t})\cdot c\in(0,1)$, where $c$
is a constant depending only on $P_0$ and $\varepsilon$, satisfies
$\Prob(P_{\infty}^{\theta}=1)>1-\varepsilon$.   \end{theorem}

Theorem~\ref{lowerbound} considers  updating rules that satisfy WBERHR and whose relative hazard rates either vanish slowly or do not vanish. It asserts that for any arbitrary lower bound on the probability of convergence to optimality, there exists  a small step-size version of the underlying updating rule such that this bound holds.

The proof of Theorem~\ref{lowerbound} can be found in Appendix~\ref{A proofs general}.  It is based on the idea of applying an increasing, concave, continuous, and bijective function $\phi:[0,1]\rightarrow[0,1]$  to the submartingale $P$ such that
$\phi(P_{0})>1-\varepsilon$. The new sequence $\left\{ \phi(P_{t})\right\} _{t\in\N_{0}}$, in general, is not a submartingale. However, WBERHR  yields a positive lower bound on the expected
differences $P_{t+1}-P_{t}$ and $\phi$ is locally
approximately linear. Applying $\phi$  to the small step-size version $P^{\theta}$,
which, in each step, only varies in a small neighborhood, corresponds
to applying an almost linear function to a submartingale with a positive
lower bound on its expected change. Hence, $\left\{ \phi(P_{t}^{\theta})\right\} _{t\in\N_{0}}$
is a submartingale. Given that $P_{\infty}^{\theta}\in\{0,1\}$,
one then obtains the statement, that is, $\mathbb{P}(P_{\infty}^{\theta}=1)>1-\varepsilon$.
The main ideas of this discussion are contained in the proof of Lemma~\ref{L_exp_inequality}.\footnote{
In Lemma~\ref{L_exp_inequality}, however, instead of constructing a new submartingale $\left\{ \phi(P_{t})\right\} _{t\in\N_{0}}$,
we are constructing a supermartingale, using the
same ideas. This modified approach simplifies the arguments for
Theorem~\ref{lowerbound} and the assertions below.}${}^,$\footnote{Taking a different point of view, small step-size updating leads to
a higher probability of converging to $1$ due to the lack of additivity
of standard deviation. With small step-size, one step is replaced
by several  steps. While the expected values of these steps are additive,
the standard deviations add up only subadditively;
thus the standard deviation--to--expected value  ratio decreases; and hence, the probability of convergence to
$1$ increases.}

Sometimes, an updating rule can directly be interpreted as a small step-size version of some fictitious updating rule. 
 The following result uses this idea:

\begin{theorem} \label{T_RE_Theory} Consider an $\{\mathcal{F}_{[0,t]}\}_{t\in\N_0}$--adapted
stochastic process $P=\{P_{t}\}_{t\in\N_{0}}$, taking values in $[0,1]$, such that the following three conditions hold:
\begin{enumerate}
\item \emph{Non-summable Relative Hazard Rates:} The sequence $P$
satisfies 
\begin{align*}
\E_{t}[P_{t+1}]-P_{t}\geq\delta_{t}P_{t}(1-P_{t})   
\end{align*}
 for some $\CF_{[0,t]}$--measurable random variable $\delta_{t}>0$,
for all $t\in\N_{0}$, with $\sum_{t=0}^{\infty}\delta_{t}=\infty$. 
 
\item \emph{WBERHR Stretchable:}
There exist a random variable $\tilde{\delta}>0$
and a sequence $\theta=\{\theta_{t}\}_{t\in\N_{0}}$ of almost surely
non-increasing $\CF_{[0,t]}$--measurable random variables ${\theta}_{t}\in(0,1]$ such that
\begin{align}
 -P_{t}\leq\frac{1}{\theta_{t}}(P_{t+1}-P_{t})\leq 1-P_t \label{small step version}
\end{align}
 and $\delta_{t}/\theta_{t}>\tilde{\delta}$
for all $t\in\mathbb{N}_{0}$. 
\item \emph{Relatively Fast Shrinking:} The stopping time $\rho$,
defined as 
\begin{align}
\rho:=\min\left\{ t\in\mathbb{N}_{0}:P_{t}\geq y\theta_{t}\right\} \text{ with }\min\emptyset:=\infty,\label{E prop1 rho}
\end{align}
 is almost surely finite for all $y\in\R$. 
\end{enumerate}
Then, $\lim_{t\uparrow\infty}P_{t}=1$. \end{theorem}

The first two conditions of Theorem~\ref{T_RE_Theory}, Non-summable Relative Hazard Rates and WBERHR Stretchable, seem natural given our previous analysis since they provide an interpretation of $P$ as a small step-size version of the performance measure of an updating rule that satisfies WBERHR. In particular, WBERHR Stretchable requires that if at each time $t$, the sequence change were $(1/\theta_t)(P_{t+1}-P_t)$ instead of $(P_{t+1}-P_t)$, then the resulting value of $P_{t+1}$ would still lie in $[0,1]$; and such an artificial sequence would satisfy WBERHR with strictly positive expected relative hazard rate (bounded by $\delta_t/\theta_t$, i.e., the original bound times $1/\theta_t$). The third condition, Relatively Fast Shrinking, requires that the fictitious compressing sequence $\theta$ tends to zero faster than $P$. Section~\ref{SS:RE} of the online appendix illustrates that the Roth-Erev learning model, for example, satisfies all these
conditions.

The proof of Theorem~\ref{T_RE_Theory} is similar to the one of
Theorem~\ref{lowerbound} and can be found in Appendix~\ref{A proofs general} as well. Recalling the informal discussion of the proof of Theorem~\ref{lowerbound}
above, we now  compress, in each period, a fictitious sequence, which
allows us to increase the concavity (and thus the value of $\phi(P_{t})$)
in each step, without losing the submartingale property of the process
$\left\{ \phi(P_{t})\right\} _{t\in\N_{0}}$. At some point in time, this
submartingale is greater than $1-\varepsilon$, for any arbitrarily given  
$\varepsilon$. This event occurs in finite time  due to the assumption of Relatively Fast Shrinking. From this
point on, the proof  follows the one of Theorem~\ref{lowerbound}.

A limitation of Theorem~\ref{lowerbound} is that it does not provide sufficient conditions for convergence to optimality
almost surely and that, for any small step-size version of the updating rule, one cannot pin down the probability of this event, unless the probability measure $\mathbb{P}$ is known. In applications, however, $\Prob$ is typically  assumed to be unknown.
These issues are taken care of by Corollary~\ref{C as}, where
we consider a reciprocally linearly decreasing compressing sequence:

\begin{corollary} \label{C as}Suppose the updating rule $\Pi$
satisfies BERHR and $P_0>0$. Then
the compressing sequence $\theta=\left\{ \theta_{t}\right\} _{t\in\N_{0}}$,
defined by $\theta_{t}=1/(t+2)$, satisfies $\Prob(P_{\infty}^{\theta}=1)=1$.
\end{corollary}

\begin{proof}  We check that the sequence $P^{\theta}$
of the statement satisfies the assumptions of Theorem~\ref{T_RE_Theory}. By \eqref{eq:slowberhr}  and the fact
that $\sum_{t=0}^{\infty}\theta_{t}=\infty$, we obtain
that $P^{\theta}$ satisfies the Non-summable Relative Hazard Rates property. The
WBERHR Stretchable property follows from the definition of $P^{\theta}$.
Finally, Lemma~\ref{L p_bound} in Appendix~\ref{A proofs general} yields the Relatively Fast Shrinking property
of $P^{\theta}$. \end{proof}

The reciprocally linearly decreasing compressing sequence $\theta$ in Corollary~\ref{C as} typically appears in stochastic approximation theory. Appendix~\ref{A proofs general} contains a discussion on the connections of this paper's results to related findings based on arguments from that literature.

\subsection{Extended framework}
\label{SS extended}

The stochastic processes studied in economics (and other sciences)  are often multivariate. For instance, in models of individual learning, these processes correspond to vectors of probabilities of choosing each action. In order to embed such models in this paper's framework, we introduce the \emph{configuration space} $\mathfrak{S}$, i.e., a convex subset of $\mathbb{R}^{D}$, where $D\in\mathbb{N}$. The elements of $\mathfrak{S}$ are called   
 \emph{configurations}.  In the setup of individual learning, $D$ is the number of actions, $\mathfrak{S}$ is the simplex of dimension $D-1$, and a configuration is a vector of probabilities of choosing each action.

Next, we map any configuration to a value in $[0,1]$, measuring its ``performance.''  We call this mapping the \emph{aggregator} $\mathfrak{A}:\mathfrak{S} \rightarrow [0,1]$ and assume that $\mathfrak{A}$ is a weighted sum of the components of the configuration. For instance, we map the vector of probabilities of choosing each action into the probability of choosing an optimal action.

With a slight misuse of notation, we redefine an updating rule as a sequence $\Pi=\left\{ \Pi_{t}\right\} _{t\in\mathbb{N}}$
of functions $\Pi_{t}:\Omega \times \mathfrak{S} \rightarrow \mathfrak{S} $ such that $\Pi_t(\cdot,\sigma)$ is  $\mathcal{F}_{[0,t]}$--measurable for all $t\in\mathbb{N}$ and $\sigma\in\mathfrak{S}$.
A pair $\left(\Pi,\mathfrak{A}\right)$  of an updating rule and an aggregator is called a \emph{system}. The performance measure $P = \{P_t\}_{t \in \N_0}$ is now iteratively defined as $P_t=\mathfrak{A}\left(\sigma_t\right)$, with $\sigma_{t+1}=\Pi_{t+1}(\sigma_t)$ for all $t\in\mathbb{N}_0$ and $\sigma_0\in\mathfrak{S}$ exogenously given.

The definitions of the properties WBERHR and BERHR generalize from updating rules to systems replacing
\eqref{bhr} of Definition~\ref{bhr} with \begin{align*}
\mathbb{E}_{t}[\mathfrak{A}(\Pi_{t+1}\left(\sigma\right))]-\mathfrak{A}\left(\sigma\right)\geq\delta_{t}\cdot \mathfrak{A}\left(\sigma\right) \left(1-\mathfrak{A}\left(\sigma\right)\right)
\end{align*}
 for all $\sigma\in\mathfrak{S}$,
and $t\in\mathbb{N}_{0}$. 
The definition of small step-size versions of updating rules is adapted to this generalization, formally replacing $p$ by $\sigma$ in \eqref{eq:slowdef}.  We say that
the system $(\Pi^{\theta},\mathfrak{A})$ is a \emph{small step-size version} of the system $(\Pi,\mathfrak{A})$ and redefine $P^{\theta}$ in an analogous way. 

With the corresponding adjustments, all  results of Section~\ref{SS convergence} generalize to this extended framework:
\begin{remark} \label{remark}
Theorem \ref{lowerbound}  and Corollary \ref{C as} hold, mutatis mutandis, replacing the updating rule $\Pi$ with the system $(\Pi,\mathfrak{A})$ and the initial performance measure $P_0$ with  $\mathfrak{A}(\sigma_0)$. Theorem~\ref{T_RE_Theory} holds as it is.
\end{remark} 
 
\section{Application to learning models}
\label{SS applications}

We provide several applications of our results in the online appendix. In a first application, we consider models of individual learning with partial information. That is, we consider an individual who every period chooses one action out of a finite set according to a probability distribution and observes a payoff realization of her choice. Upon observing this realization, she adjusts the probability of choosing each action according to a function mapping (potentially all) past realizations of her choices and obtained payoffs to the revised probability of choosing each action. In such a setup  \Citet{Borgersetal2004} identify conditions on these mappings such that, in every period, the conditional expected value of the change in the probability of choosing expected-payoff maximizing actions is positive. The conditions they identify yield learning models  that satisfy the BERHR property. Their analysis, however, focuses on the change in the  probability of choosing each action from one period to the next and hence, is silent about the asymptotic properties of these models. Instead, by explicitly extending their framework to an infinite horizon, we show that  such learning models converge to payoff maximizing actions: (i) with high probability,  when they exhibit small step-size (applying Theorem~\ref{lowerbound}) and (ii) almost surely,  when they exhibit linear shrinking step-size such as, for instance, one of the learning models considered in \citet{SarinVahid2004} (applying Corollary~\ref{C as}). We also provide a similar construction
in a setting where individuals receive full information, i.e., observe 
both obtained and forgone payoffs. For details, see Section~\ref{SS:individual} in the online appendix.

In a second application, we analyze  models of social learning. We now consider a population of individuals who, in every period, choose one action within a finite set. Individuals adjust the probability of choosing each action upon observing their own payoff and also the actions chosen and payoffs obtained by a random sample of some of the other individuals. Each individual's revised probability of choosing each action is a weighted average of its previous probability and an \emph{imitation component} that only places probability on the actions that the individual observed, as a function of the payoffs they yielded. The weight of the imitation component is called the \emph{imitation rate}. \Citet{Schlag1998} considers a version of this model such that only one other individual's action and payoff is sampled and the imitation rate is $1$, i.e., each individual chooses with positive probability only the action she chose or the action chosen by the individual she observed. He provides conditions on the imitation component that allow the average payoff of the population, in expected value, to increase in every  period.\footnote{These conditions require that the difference between the probability that an individual who chose action $a$ switches to  action $b$ and the probability that an individual makes the opposite switch is an  increasing linear function of the difference between the payoffs yielded by $b$ and $a$. } For finite populations, however, \Citet{Schlag1998}  finds that, with positive probability, all individuals  converge to choosing non-optimal actions. By instead assuming linearly decreasing imitation rates, we can apply Corollary~\ref{C as} to prove that the event in which the whole population converge to choose optimal actions occurs almost surely.\footnote{In \Citet{OyarzunRuf2009} we generalize the conditions in \Citet{Schlag1998}: if the function describing the net switching from $a$ to $b$ is strictly increasing  in the payoff of $b$ and strictly decreasing in the payoff of $a$ (and satisfies some symmetry condition),  then the fraction of the population who choose first-order stochastically dominant actions is strictly increasing in expectation and the convergence results of this paper can  be applied to that setup as well.} For details, see Section~\ref{social} in the online appendix.

In a third application, we consider the  Roth-Erev model of individual learning (see  \citet{ErevRoth1998}). In this model, in every period,  each action has an ``attraction''  corresponding to the accumulated sum of payoffs that this action has yielded when it has been chosen, and the probability of choosing each action is proportional to its attraction. This model has embedded linearly shrinking step-size, therefore  we can use Theorem~\ref{T_RE_Theory} to analyze its asymptotical properties.  \Citet{Beggs2005} and \citet{HopkinsPosch2005}  prove that this model converges to payoff maximization using  arguments based on stochastic approximation. In contrast, the proof we provide in   the online appendix builds on the properties of the expected relative hazard rates of this learning model. Hence, our results provide a different interpretation of the convergence property of this model.\footnote{\Citet{Beggs2005} and \citet{HopkinsPosch2005} assume that payoffs are bounded away from zero, which is not required in our proof.}  For details, see Section~\ref{SS:RE} in the online appendix.

\section{Discussion}

The  analysis of systems that satisfy WBERHR or BERHR can be the starting
point for the study of slightly more complex dynamics. There are many
other models in the literature with similar characteristics to those
considered here that do not satisfy these properties. One example
is the model of word-of-mouth social learning in \citet{EllisonFudenberg1995}.
In their model, individuals sample $n\in\mathbb{N}$ other individuals
out of a continuum population and choose the action that has the highest
average payoff in their observed sample. Aggregate shocks (on top
of individual specific shocks) of the payoffs yielded by the two available
actions allow for randomness despite of the population's cardinality.
For $n=1$, their model satisfies BERHR, and hence, their findings
are recovered by our results. In particular, the population may ``herd''
to the action with the lowest expected payoff with positive probability,
and this probability goes to zero when there is enough inertia (which
is equivalent to shrinking step-size in our analysis). For $n>2$, however,
their model does not satisfy WBERHR and thus our results tell us
nothing about the asymptotic properties of their model. Future research
could study related conditions on these systems that make it possible to analyze
the asymptotic properties of the models in a general framework encompassing
their findings for those cases.

Another possible extension is the study of properties of systems that
satisfy (W)BERHR in games. \citet{Beggs2005} proves that the Roth-Erev
learning model leads individuals to converge to play with zero-probability
actions eliminated by iterated deletion of dominated strategies. \citet{TarresVandekerkhove2012} show that the two-armed bandit algorithm converges to the arm that is optimal in average, even if its expected payoff is smaller in some periods. Their results suggest that our approach could be extended to analyze setups in which the set of optimal actions may not be the same in each single period, as it often is the case in learning in games.
This topic deserves further
attention in the future.

\citet{LambertonPages2008a, LambertonPages2008b} provide the rates of convergence for the two-armed bandit algorithm. Indeed, the trade-off between speed and the probability of achieving convergence to optimality
is of particular interest in the literature of learning automata
(see, e.g., \citet{NarendraThathachar}) and hence, worth of further
study in the setup of this paper.   


In our analysis of the properties of the dynamics of choices in social learning models, our sampling assumptions may seem restrictive in some setups. For instance, observability may rule out
network structures in which individuals may sample some other individuals in the network with zero probability (see, e.g., \citet{BalaGoyal1998}). It is intuitive, however, that the choices of  individuals who are not sampled may be observed, after a number of periods, provided that there is a path of individuals connecting the individual who chose an action and another who could choose that action later via imitation. Analyzing the dynamics of the performance  measure in such structures would require developing further the constructions provided above.  This topic is left for future research.

\appendix


\section{Proofs of the convergence results in Section~\ref{SS convergence}}   \label{A proofs general}
In this appendix, we provide the proofs of the statements in Section~\ref{SS convergence}.
Although the models we analyze resemble a typical setup of stochastic approximation theory in the spirit of \Citet{RobbinsMonro1951}, \Citet{KieferWolfowitz1952},  \Citet{KushnerClark1978}, and \citet{KushnerYin2003}, the proofs in this appendix do not rely on standard techniques developed in that literature.  For an excellent overview of that literature, we refer the reader to \citet{Benaim1999} and the references therein. \Citet{FudenbergKreps1993}, \Citet{HopkinsPosch2005}, and  \citet{BenaimFaure2012} provide examples where stochastic approximation techniques have been fruitfully applied to economic learning models. A similar approach to analyze bandit problems is developed by  
\citet{Lamberton_etal2004}, \citet{LambertonPages2008b}, and \citet{TarresVandekerkhove2012}.  In the setup of our paper, we found  that arguing from first principles and extending results in  \citet{Norman1968b} and \Citet{LT1976}  was tractable for the generality of our statements.

We start with  the core insight for the proofs of the statements in Section~\ref{SS convergence}. Here, we strongly rely on the positivity of the lower bound sequence $\delta$ in the definition of the (W)BERHR condition.  The following lemma is inspired by the ideas in \citet{Norman1968b} and \Citet{LT1976}:\begin{lemma} \label{L_exp_inequality} If the updating rule $\Pi$
satisfies \eqref{so1} for some 
$p\in[0,1]$, $t\in\mathbb{N}_{0}$, and $\CF_{[0,t]}$--measurable
$\delta_{t}\geq0$ then 
\begin{align*}
\mathbb{E}_{t}\left[e^{-\frac{\gamma}{\theta_{t}}\Pi_{t+1}^{\theta}(p)}\right]\leq e^{-\frac{\gamma}{\theta_{t}}p}
\end{align*}
 for all $\CF_{[0,t]}$--measurable $\gamma\in[0,1\wedge\delta_{t}]$
and $\CF_{[0,t]}$--measurable $\theta_{t}\in(0,1]$. \end{lemma}
\begin{proof} We only need to show the statement for $\gamma>0$.
Thus, without loss of generality, assume that the event $\{\delta_{t}>0\}$ occurs.
Define the function $G_{\gamma}:[0,1]\rightarrow\R$ for all $\gamma\in(0,1\wedge\delta_{t}]$
by 
\begin{equation}
G_{\gamma}(z):=z+\delta_{t}z(1-z)-\frac{1-e^{-\gamma z}}{1-e^{-\gamma}}, \label{E G def}
\end{equation} 
and observe that $G_{\gamma}(0)=G_{\gamma}(1)=0$ and that
\[
\frac{\partial^{2}}{\partial z^{2}}G_{\gamma}(z)=-2\delta_{t}+\frac{\gamma^{2}}{1-e^{-\gamma}}\cdot e^{-\gamma z}\leq-2\delta_{t}+\frac{\gamma}{1-e^{-\gamma}}\cdot\gamma\leq2(\gamma-\delta_{t})\leq0,
\]
 since $\gamma/(1-e^{-\gamma})<2$ for all $\gamma\in(0,1\wedge\delta_{t}]$. This yields $G_{\gamma}(z)\geq0$.
Similarly, we see that 
\begin{align}
\widetilde{G}_{\gamma}(z):=z-\frac{1-e^{-\gamma z}}{1-e^{-\gamma}}\leq0\label{E tildeG ineq}
\end{align}
 for all $z\in[0,1]$ and $\gamma\in(0,1]$, which yields that 
\begin{align*}
\E_{t}\left[\frac{1-e^{-\gamma\Pi_{t+1}(p)}}{1-e^{-\gamma}}\right]\geq\E_{t}\left[\Pi_{t+1}(p)\right]\geq G_{\gamma}(p)+\frac{1-e^{-\gamma p}}{1-e^{-\gamma}}\geq\frac{1-e^{-\gamma p}}{1-e^{-\gamma}},
\end{align*}
 for all $p\in[0,1]$, where the first inequality follows from \eqref{E tildeG ineq} with
$z=\Pi_{t+1}(p)$ and the second inequality follows
from \eqref{E G def} and \eqref{so1}. This yields  
\begin{align*}
\mathbb{E}_{t}\left[e^{-\frac{\gamma}{\theta_{t}}\left(\Pi_{t+1}^{\theta}(p)-p\right)}\right]=\mathbb{E}_{t}\left[e^{-\gamma\left(\Pi_{t+1}(p)-p\right)}\right]\leq1,
\end{align*}
which proves the statement. 
\end{proof}

Now, we provide the proof of Theorem~\ref{lowerbound} by applying the previous lemma:

\begin{proof}[Proof of Theorem~\ref{lowerbound}] Fix the smallest integer
$\widetilde{\gamma}=\widetilde{\gamma}(P_{0},\varepsilon)$ such that
$e^{-\widetilde{\gamma}P_{0}}<\varepsilon.$ Define the compressing
sequence $\{\theta\}_{t\in\N_0}$ by $\theta_{t}:=(\delta_{t}\wedge1)/\widetilde{\gamma}\leq1$, observe that $\sum_{t=0}^{\infty}\theta_{t}\delta_{t}=\infty$, and define the
process $M=\{M_{t}\}_{t\in\N_{0}}$ by $M_{t}:=e^{-\widetilde{\gamma}P_{t}^{\theta}}$
for all $t\in\N_{0}$. We start by observing that $M$ is a supermartingale since, for fixed
$t\in\N_{0}$, we have  
\begin{align*}
\E_{t}\left[M_{t+1}\right] & =\mathbb{E}_{t}\left[e^{-\widetilde{\gamma}P_{t+1}^{ \theta}}\right]=\mathbb{E}_{t}\left[e^{-\frac{(\delta_{t}\wedge1) }{\theta_{t} }P_{t+1}^{ \theta}}\right]\leq e^{-\frac{(\delta_{t}\wedge1) }{\theta_{t} }P_{t}^{ \theta}}=M_{t},
\end{align*}
 where the inequality follows from Lemma~\ref{L_exp_inequality}.
Thus, $M_{t}$ converges to some random variable $M_{\infty}\in[0,1]$
and we obtain 
\begin{align*}
\mathbb{P}\left(P_{\infty}^{ \theta}=1\right) & =\mathbb{E}[P_{\infty}^{ \theta}]\geq1-\E[M_{\infty}]\geq1-M_{0}\geq1-\varepsilon,
\end{align*}
 where we have used \eqref{eq:slowberhr} and Lemma~\ref{support} in
the first equality.  \end{proof}

The proof of Theorem~\ref{T_RE_Theory} is similar:

\begin{proof}[Proof of Theorem~\ref{T_RE_Theory}] Fix $\varepsilon\in(0,1)$
and observe that there exists a constant $\delta\in(0,1)$ such that
the event $A:=\{\widetilde{\delta}\geq\delta\}\subset\{\inf_{t\in\N_{0}}\{\delta_{t}/\theta_{t}\}\geq\delta\}$
satisfies $\Prob(A)\geq1-\varepsilon/2$. Define $y:=-\log(\varepsilon/2)/\delta>0$
and the process $M=\{M_{t}\}_{t\in\N_{0}}$ by 
\begin{align*}
M_{t}:=e^{-\frac{\delta}{\theta_{t\wedge\rho}}P_{t}}\1_{\{\min_{u\in\{0,\ldots,t\}}\{\delta_{u}/\theta_{u}\}\geq\delta\}}
\end{align*}
 for all $t\in\N_{0}$, where the stopping time $\rho$ is given in
\eqref{E prop1 rho}. As in the proof of Theorem~\ref{lowerbound}
we start by showing that $M$ is a supermartingale. Towards this end,
fix some $t\in\N_{0}$ and assume, without loss of generality, that
we are on the event $\{\min_{u\in\{0,\ldots,t\}}\{\delta_{u}/\theta_{u}\}\geq\delta\}$.
Define now  $\widetilde{P}_{t}=P_{t}$ and $\widehat{P}_{t+1}=\widetilde{P}_{t}+(P_{t+1}-P_{t})/\theta_{t}\in[0,1]$ and
observe that $\E_t[\widehat{P}_{t+1}]-\widetilde{P}_{t}\geq\delta\widetilde{P}_{t}(1-\widetilde{P}_{t})$. We then obtain that 
\begin{align*}
\E_{t}\left[M_{t+1}\right] & \leq\mathbb{E}_{t}\left[e^{-\frac{\delta}{\theta_{(t+1)\wedge\rho}}P_{t+1}}\right]\leq\mathbb{E}_{t}\left[e^{-\frac{\delta}{\theta_{t\wedge\rho}}P_{t+1}}\right]=\mathbb{E}_{t}\left[e^{-\frac{\delta\theta_{t}/\theta_{t\wedge\rho}}{\theta_{t}}P_{t+1}}\right]\leq e^{-\frac{\delta}{\theta_{t\wedge\rho}}P_{t}}=M_{t},
\end{align*}
 where we have used Lemma~\ref{L_exp_inequality}, in which we interpret
$(P_{t},P_{t+1})$ as the ``compressed version" of $(\widetilde{P}_{t},\widehat{P}_{t+1})$. Thus, as in the proof of Theorem~\ref{lowerbound}, $M_{t}$ converges to some random variable $M_{\infty}\in[0,1]$.

As $\rho<\infty$ almost surely by assumption, we obtain from an argument similar to the one in Lemma~\ref{support} that 
\begin{align*}
\mathbb{P}\left(P_{\infty}=1\right) & =\mathbb{E}[P_{\infty}]\geq1-\E\left[e^{-\frac{\delta}{\theta{_{\rho}}}P_{\infty}}\right]\geq1-\E[M_{\infty}]-\Prob(A^{C})\geq\Prob(A)-\E[M_{\rho}]\\
 & \geq\Prob(A)-\E\left[e^{-\frac{\delta}{\theta_{\rho}}P_{\rho}^{\theta}}\right]\geq1-\frac{\varepsilon}{2}-e^{-\delta y}=1-\varepsilon,
\end{align*}
 similarly to the proof of Theorem~\ref{lowerbound}, where $A^{C}:=\Omega\setminus A$.
As $\varepsilon$ was chosen arbitrarily, we obtain the statement. \end{proof}

For the proof of Corollary~\ref{C as}, we make the following useful observation:

\begin{lemma} \label{L pmin} For an updating rule $\Pi$
with corresponding performance sequence $\{P_{t}\}_{t\in\N_{0}}$,
we have $P_{t}^{\theta}\geq P_{0}/(t+1)$ for all $t\in\N_{0}$ for
the compressing sequence $\theta$ defined by $\theta_{t}=1/(t+2)$. \end{lemma}
\begin{proof} Assume that we have shown $P_{t}^{\theta}\geq P_{0}/(t+1)$
for some fixed $t\in\N_{0}$. Then 
\begin{align*}
P_{t+1}^{\theta}=P_{t}^{\theta}+\frac{1}{t+2}\left(\Pi_{t+1}\left(P_t^{\theta}\right)-P_{t}^{\theta}\right)\geq P_{t}^{\theta}-\frac{P_{t}^{\theta}}{t+2}=\frac{(t+1)P_{t}^{\theta}}{t+2}\geq\frac{P_{0}}{t+2},
\end{align*}
 and the statement follows by induction. \end{proof}

Simple computations then yield the following core conclusion:
\begin{lemma} \label{L p_bound} Let the updating rule $\Pi$
satisfy WBERHR with performance sequence $\{P_{t}\}_{t\in\N_{0}}$
and $P_{0}>0$ and let $\theta$ denote the compressing sequence defined
by $\theta_{t}:=1/(t+2)$ for all $t\in\N_{0}$. Fix $y>0$ and define
the stopping time $\rho$ as 
\begin{align}
\rho:=\min\left\{ t\in\mathbb{N}_{0}:P_{t}^{\theta}\geq\frac{y}{t+2}\right\} \text{ with }\min\emptyset:=\infty.\label{E tau2}
\end{align}
 Then $\Prob(\rho<\infty)=1$. \end{lemma} \begin{proof} We shall
show the statement with $y$ in \eqref{E tau2} replaced by $NP_{0}$,
where $N=[y/P_{0}]+1$ with $[\cdot]$ here denoting the largest integer
smaller than the argument. Define $N_{k}:=N^{2k}-1$ for all $k\in\N_{0}$.
Since $\Prob(\rho>t)$ is non-increasing in $t$, it is sufficient
to show that $\Prob(\rho>N_{k})\downarrow 0$ as $k\uparrow \infty$.
This follows if there exists some constant $c\in[0,1)$ independent
of $k$ such that 
\begin{align}
p_{1}:=\Prob(\rho>N_{k+1})\leq c\Prob(\rho>N_{k})=:cp_{0}\label{E L_p1_p0}
\end{align}
 for all $k\in\N$. Towards this end, Lemma~\ref{L pmin} and the
submartingale property of $\left\{ P_{t}\right\} _{t\in\N_{0}}$ yield 
\begin{align*}
\frac{P_{0}}{N^{2k}}p_{0} & \leq\E[P_{N_{k}}^{\theta}\1_{\{\rho>N_{k}\}}]\leq\E\left[\E_{N_{k}}[P_{N_{k+1}\wedge\rho}^{\theta}]\1_{\{\rho>N_{k}\}}\right]=\E\left[P_{N_{k+1}}^{\theta}\1_{\{\rho>N_{k+1}\}}\right]+\E\left[P_{\rho}^{\theta}\1_{\{N_{k}<\rho\leq N_{k+1}\}}\right]\\
 & \leq\frac{NP_{0}}{N^{2(k+1)}+1}p_{1}+\left(\frac{NP_{0}}{N^{2k}+1}+\frac{1}{N^{2k}+1}\right)(p_{0}-p_{1})\leq\frac{P_{0}}{N^{2k}}\left(\frac{p_{1}}{N}+\left(N+\frac{1}{P_{0}}\right)(p_{0}-p_{1})\right)
\end{align*}
 since $P_{t+1}^{\theta}\leq P_{t}^{\theta}+\theta_{t}$ and $p_{0}\geq p_{1}$.
 Sorting terms, we have 
\begin{align*}
p_{1}\left(N+\frac{1}{P_{0}}-\frac{1}{N}\right) & \leq p_{0}\left(N+\frac{1}{P_{0}}-1\right)
\end{align*}
 and, thus, \eqref{E L_p1_p0}. \end{proof}

\citet{Lamberton_etal2004} prove a generalized version of Corollary~\ref{C as} for the special case of the two-armed bandit algorithm in their Theorem~1(c) and Corollary~2. Their argument can be summarized in three steps. First, they establish a version of Lemma~\ref{support} above. Then, via a coupling argument, they relate the   
two-armed bandit algorithm to another one where both arms have the same distribution. Third, they use martingale methods to prove that the later algorithm does not get trapped in zero.  Putting the three steps together they   obtain a version of Corollary~\ref{C as} above.

In contrast, the proofs provided here rely deeply on the bounds on the relative hazard rates.  In particular, Lemma~\ref{L_exp_inequality}  is related to the fact that a concave function of a submartingale is still a submartingale, as long as the original submartingale only moves in small steps, and is ``sufficiently drifted,'' which the BERHR condition guarantees.  Lemma~\ref{L p_bound}, which is used in the proof of Corollary~\ref{C as}, can be interpreted as the statement that the performance sequence does not get trapped in zero, similar to the third step in \citet{Lamberton_etal2004}.

\setlength{\bibsep}{1pt} {\small \bibliographystyle{apalike}
\bibliography{CO_JR_2013}

\end{document}